\documentclass[a4paper,12pt]{article}

\usepackage[leqno]{amsmath}
\usepackage{amsthm,amsfonts}
\usepackage{enumerate}
\theoremstyle{plain}
\newtheorem{lemma}{Lemma}[section]
\newtheorem{theorem}[lemma]{Theorem}
\newtheorem{prop}[lemma]{Proposition}

\theoremstyle{definition}

\newtheorem{remark}[lemma]{Remark}

\newcommand{\Ho}{\mathfrak{H}}
\title{Fat Hoffman graphs 
with smallest eigenvalue greater than $-3$}

\author{
{\sc Akihiro MUNEMASA $^\dagger$}\\
[1ex]
{\small 
Graduate School of Information Sciences,} \\
{\small Tohoku University, 
Sendai 980-8579, Japan} \\
{\small 
{\it E-mail address}: {\tt munemasa@math.is.tohoku.ac.jp}}\\
\\
{\sc Yoshio SANO}
\thanks{This work was supported by JSPS KAKENHI grant number 25887007.}\\
[1ex]
{\small
Division of Information Engineering,} \\
{\small 
Faculty of Engineering, Information and Systems,} \\
{\small University of Tsukuba, 
Ibaraki 305-8573, 
Japan} \\
{\small 
{\it E-mail address}: {\tt sano@cs.tsukuba.ac.jp}}\\
\\
{\sc Tetsuji TANIGUCHI}
\thanks{This work was supported by JSPS KAKENHI grant number 25400217.}\\
[1ex]
{\small 
Matsue College of Technology, 
Matsue 690-8518, Japan} \\
{\small 
{\it E-mail address}: {\tt tetsuzit@matsue-ct.ac.jp}}
}

\date{}

\begin{document}

\maketitle

\begin{abstract}
In this paper, we 
give a combinatorial characterization
of the special graphs of fat Hoffman graphs 
containing $\mathfrak{K}_{1,2}$ 
with smallest eigenvalue greater than $-3$, 
where $\mathfrak{K}_{1,2}$ is the Hoffman graph 
having one slim vertex and two fat vertices. 
\end{abstract}

\noindent
{\bf Keywords:}
graph eigenvalue; 
line graph; 
block graph; 
signed graph. 
\\

\noindent
{\bf 2010 Mathematics Subject Classification:}
05C50, 05C75

\newpage

\section{Introduction}\label{sec:001}

In the field of Spectral Graph Theory, one of the important research problems 
is to characterize graphs with bounded smallest eigenvalue. 
In 1976, using root systems,
P.~J.~Cameron, J.~M.~Goethals, J.~J.~Seidel, and E.~E.~Shult
\cite{root}
characterized graphs whose adjacency matrices have smallest eigenvalue 
at least $-2$. 
Their results revealed that graphs with smallest eigenvalue at least $-2$ 
are generalized line graphs, except a finite number of graphs represented 
by the root system $E_8$. We refer the reader to the monograph \cite{new},
for a complete account of this theory.
In 1977, A.~J.~Hoffman \cite{hoffman1} studied graphs whose adjacency matrices 
have smallest eigenvalue at least $-1-\sqrt{2}$ 
by using a technique of adding cliques to graphs.
In 1995, R.~Woo and A.~Neumaier \cite{hlg} formulated Hoffman's idea 
by introducing the notion of Hoffman graphs and generalizations 
of line graphs. 
Hoffman graphs were 
subsequently studied in
\cite{JKMT,MST1,paperI,paperII,Yu}. 
In particular, 
H.~J.~Jang, J.~Koolen, A.~Munemasa, and T.~Taniguchi~\cite{JKMT} 
proposed a scheme to classify
fat indecomposable Hoffman graphs with smallest eigenvalue at least $-3$.
The present paper completes a partial case of this scheme.
While
there are quite a few Hoffman graphs with smallest eigenvalue
at most $-3$ (see Section~\ref{sec:-3}), there are strong restrictions
on Hoffman graphs with smallest eigenvalue greater than $-3$. 
The counterpart of this problem for ordinary graphs is
the classification of graphs with smallest eigenvalue greater than
$-2$, given by \cite{DC79}. In this paper, we consider
fat indecomposable Hoffman graphs with smallest eigenvalue
greater than $-3$, under the additional assumption that there
exists a slim vertex with two fat neighbors.

Our result can also be regarded as a reformulation of 
a classical result of Hoffman~\cite{hoffman0}
in terms of Hoffman graphs.
Let $\hat{A}(G,v^*)$ denote the adjacency matrix of a graph $G$,
modified by putting $-1$ in the diagonal position corresponding to 
a vertex $v^*$.
Hoffman \cite[Lemma 2.1]{hoffman0} has shown that
$\hat{A}(L(T),e)$ has smallest eigenvalue greater than $-2$ whenever
$e$ is an end edge of a tree $T$, where $L(T)$ denotes the line graph
of $T$. Moreover, under a conjecturally redundant assumption, 
Hoffman \cite[Lemma 2.2]{hoffman0} has shown that
the smallest eigenvalue of $\hat{A}(L(T),e)$ is a limit point of
the set of smallest eigenvalues of graphs. 
Denoting by $\lambda_{\min}(A)$ the smallest eigenvalue of a real
symmetric matrix $A$, this implies that
$\lambda_{\min}(\hat{A}(L(T),e))-1$ is also a limit point of
the set of smallest eigenvalues of graphs. 
In fact, $\lambda_{\min}(\hat{A}(L(T),e))-1$ can be regarded as the 
smallest eigenvalue of a Hoffman graph, and by Hoffman's limit theorem
\cite[Proposition 3.1]{hoffman1} (see also \cite{hoffmanSIAM, JKMT}),
it is a limit point of the set of smallest eigenvalues of graphs. 
In this way, we obtain limit points of smallest eigenvalues of graphs
with smallest eigenvalue greater than $-3$.

The goal of this paper is to characterize 
the special graphs of fat indecomposable
Hoffman graphs with smallest eigenvalue greater than $-3$
containing a slim vertex having two fat neighbors. 
As a consequence, we show in Theorem~\ref{thm:5} that,
if the smallest eigenvalue of $\hat{A}(G,v^*)$ is greater than $-2$, 
then $G$ is the line graph of a tree $T$
and $v^*$ corresponds to an end edge of $T$. 

The organization of the paper is as follows. In Section~\ref{sec:pre},
we give basic results on Hoffman graphs and block graphs which are needed
in later sections. In Section~\ref{sec:-3}, we show that various
Hoffman graphs have smallest eigenvalue at most $-3$. These graphs
will play a role of forbidden subgraphs for the family of fat Hoffman
graphs with smallest eigenvalue greater than $-3$. In 
Section~\ref{sec:main}, we give our main theorem which
characterizes the special graphs of fat indecomposable
Hoffman graphs with smallest eigenvalue greater than $-3$
containing a slim vertex having two fat neighbors. 
Finally, in Section~\ref{sec:conc}, we give an extension
of a lemma of Hoffman \cite{hoffman0}
about the smallest eigenvalue of the
modified adjacency matrix of a graph.

\section{Preliminaries}\label{sec:pre}
\subsection{Hoffman graphs}

A \emph{Hoffman graph} $\mathfrak{H}$ is a pair 
of a (simple undirected) graph $(V(\Ho),E(\Ho))$ 
and a distinguished coclique $F \subseteq V(\Ho)$. 
A vertex in $F$ is called a \emph{fat vertex} 
and a vertex in $V(\Ho) \setminus F$ is called a \emph{slim vertex}. 
We denote $F$ and $V(\Ho) \setminus F$ by $V^s(\Ho)$ and $V^f(\Ho)$,
respectively.
In this paper, we assume that no fat vertex is isolated. 

For a vertex $x$ of a Hoffman graph $\Ho$, 
a \emph{slim neighbor} (resp.\ a \emph{fat neighbor}) of $x$ in $\Ho$ 
is a slim vertex (resp.\ a fat vertex) $y$ of $\Ho$ 
such that $\{x,y\}$ is an edge of $\Ho$. 
We denote by the set of slim neighbors (resp.\ fat neighbors) 
of $x$ in $\Ho$ by $N^s_{\Ho}(x)$ (resp.\ $N^f_{\Ho}(x)$).
A Hoffman graph $\Ho$ is said to be \emph{fat} 
if every slim vertex of $\Ho$ has a fat neighbor, 
and $\Ho$ is said to be \emph{slim} 
if $\Ho$ has no fat vertex. 

Two Hoffman graphs $\Ho$ and $\Ho'$
are said to be \emph{isomorphic} 
if there exists a bijection $\phi:V(\Ho) \to V(\Ho')$
such that $\phi(V^s(\Ho))= V^s(\Ho')$, 
$\phi(V^f(\Ho))= V^f(\Ho')$, and 
$\{x,y\} \in E(\Ho)$ if and only if 
$\{\phi(x),\phi(y) \} \in E(\Ho)$. 
A Hoffman graph $\Ho'$ is called
an \emph{induced Hoffman subgraph} 
of a Hoffman graph $\Ho$ 
if $V^s(\Ho') \subseteq V^s(\Ho)$, 
$V^f(\Ho') \subseteq V^f(\Ho)$, 
and 
$E(\Ho')=\{\{x,y\} \in E(\Ho) \mid x,y \in V(\Ho')\}$. 

Let 
\[
A(\Ho)=
\begin{pmatrix}
A^s(\Ho) & C(\Ho) \\
C(\Ho)^T & O
\end{pmatrix}
\]
be the adjacency matrix of a Hoffman graph $\mathfrak{H}$,
in a labeling in which the slim vertices come first and 
the fat vertices come last.
The \emph{eigenvalues} of $\Ho$ are 
defined to be
the eigenvalues of the
real symmetric matrix 
\[
B(\Ho)=A^s(\Ho)-C(\Ho)C(\Ho)^T. 
\]
We denote the smallest eigenvalue of $B(\Ho)$ by $\lambda_{\min}(\Ho)$.

\begin{lemma}[{\cite[Corollary 3.3]{hlg}}]\label{lm:001}
If $\Ho'$ is an induced Hoffman subgraph of a Hoffman graph $\mathfrak{H}$, 
then $\lambda_{\min}(\Ho')\geq\lambda_{\min}(\mathfrak{H})$ holds. 
\end{lemma}

A {\em decomposition} of a Hoffman graph $\Ho$ 
is a family $\{\Ho^i\}_{i=1}^n$ of non-empty
induced Hoffman subgraphs of $\Ho$
satisfying the following conditions:
\begin{itemize}
\item[{\rm (i)}] 
$V(\Ho)=\bigcup_{i=1}^n V(\Ho^i)$;
\item[{\rm (ii)}] 
$V^s(\Ho^i) \cap V^s(\Ho^j)=\emptyset$ if $i\neq j$;
\item[{\rm (iii)}] 
For each $x\in V^s(\Ho^i)$, 
$N^f_{\Ho}(x) \subseteq V^f(\Ho^i)$
\item[{\rm (iv)}] 
If $x \in V^s(\Ho^i)$, $y\in V^s(\Ho^j)$, and $i \neq j$, 
then $|N^f_{\Ho}(x) \cap N^f_{\Ho}(y)| \leq 1$, 
and $|N^f_{\Ho}(x) \cap N^f_{\Ho}(y)| = 1$ 
if and only if $\{x, y\} \in E(\Ho)$.
\end{itemize}
A Hoffman graph $\Ho$ is said to be {\em decomposable} 
if 
$\Ho$ has a decomposition $\{\Ho^i\}_{i=1}^n$ with $n \geq 2$, 
and $\Ho$ is said to be {\em indecomposable} 
if $\Ho$ is not decomposable.

\begin{lemma}[{\cite[Lemma 2.12]{JKMT}}]\label{lm:002}
If a Hoffman graph $\Ho$ has a decomposition $\{\Ho^i\}_{i=1}^n$,
then 
$\lambda_{\min}(\Ho)=\min\{\lambda_{\min}(\Ho^i) \mid 1 \leq i \leq n\}$. 
\end{lemma} 

Let $\Ho$ be a Hoffman graph and 
let $m$ and $N$ be positive integers. 
A \emph{reduced representation of norm $m$} of $\Ho$ 
is a map $\psi:\ V^s(\Ho)\to \mathbb{R}^N$
such that 
\[
\langle \psi (x), \psi (y) \rangle =
\begin{cases}
m - |N^f_{\Ho}(x)| & 
\text{if $x=y$,} \\
1 - |N^f_{\Ho}(x) \cap N^f_{\Ho}(y)| & 
\text{if $\{x, y\} \in E(\Ho)$,}\\
-|N^f_{\Ho}(x) \cap N^f_{\Ho}(y)| & 
\text{otherwise}, \\
\end{cases}
\]
where $\langle \ , \ \rangle$ denotes the standard inner product on
$\mathbb{R}^N$. It follows immediately from the definitions that
$B(\Ho)_{x,y}=\langle \psi (x), \psi (y) \rangle $ holds for
any distinct slim vertices $x,y$.

\begin{lemma}[{\cite[Theorem 2.8]{JKMT}}]\label{lm:redrep} 
Let $\Ho$ be a Hoffman graph 
and let $m$ be a positive integer. 
Then, $\lambda_{\min}(\Ho) \geq -m$ 
if and only if 
$\Ho$ has a reduced representation of norm $m$. 
\end{lemma}

\begin{lemma}[{\cite[Lemma 3.6]{JKMT}}]\label{lm:offB} 
Let $\Ho$ be a fat Hoffman graph.
If the matrix $B(\Ho)$ has some off-diagonal entry less than $-1$, 
then $\lambda_{\min}(\Ho) < -3$.
\end{lemma}

\begin{lemma}\label{lm:redrep2} 
Let $\Ho$ be a Hoffman graph 
and let $m$ be a positive integer. 
Then, $\lambda_{\min}(\Ho) > -m$ 
if and only if 
$\Ho$ has a reduced representation $\psi$ of norm $m$ 
such that $\{\psi(v) \mid v \in V^s(\Ho)\}$ 
is linearly independent. 
\end{lemma}

\begin{proof}
See the proof of {\cite[Theorem 2.8]{JKMT}}. 
\end{proof}

An \emph{edge-signed graph} $\mathcal{S}$ 
is a triple 
$(V(\mathcal{S}), E^{+}(\mathcal{S}), E^{-}(\mathcal{S}))$
of a set $V(\mathcal{S})$ of \emph{vertices}, 
a set $E^{+}(\mathcal{S})$ of $2$-subsets  
of $V(\mathcal{S})$ (called \emph{$(+)$-edges}), and 
a set $E^{-}(\mathcal{S})$ of $2$-subsets  
of $V(\mathcal{S})$ (called \emph{$(-)$-edges}) 
such that  
$E^{+}(\mathcal{S}) \cap E^{-}(\mathcal{S}) = \emptyset$. 
The \emph{underlying graph} $U(\mathcal{S})$ of 
an edge-signed graph $\mathcal{S}$ 
is the (unsigned) graph 
$(V(\mathcal{S}),E^+(\mathcal{S}) \cup E^-(\mathcal{S}))$.

An edge-signed graph 
$\mathcal{S}'$ 
is called an {\em induced edge-signed subgraph} 
of an edge-signed graph 
$\mathcal{S}$ 
if $V(\mathcal{S}') \subseteq V(\mathcal{S})$, 
$E^{\pm}(\mathcal{S}')=\{\{x,y\} \in E^{\pm}(\mathcal{S}) 
\mid x,y \in V(\mathcal{S}')\}$.
Two edge-signed graphs $\mathcal{S}$ and $\mathcal{S}'$ 
are said to be \emph{isomorphic} 
if there exists a bijection $\phi: V(\mathcal{S}) \to V(\mathcal{S}')$ 
such that $\{u,v\} \in E^\pm(\mathcal{S})$ if and only if 
$\{\phi(u), \phi(v) \} \in E^\pm(\mathcal{S}')$.

The \emph{special graph} of a Hoffman graph $\Ho$ is 
the edge-signed graph 
$\mathcal{S}(\Ho)$ 
defined by 
$V(\mathcal{S}(\Ho))=V^s(\Ho)$ and 
\begin{align*}
E^+(\mathcal{S}(\Ho)) &= \{ \{u,v\} \mid 
u,v \in V^s(\Ho), 
u \neq v, 
\{u,v\} \in E(\Ho), 
N^f_{\Ho}(u) \cap N^f_{\Ho}(v) = \emptyset \}, \\
E^-(\mathcal{S}(\Ho)) &= \{ \{u,v\} \mid 
u,v\in V^s(\Ho), 
u \neq v, 
\{u,v\} \notin E(\Ho), 
N^f_{\Ho}(u) \cap N^f_{\Ho}(v) \neq \emptyset \}. 
\end{align*}

\begin{lemma}[{\cite[Lemma 3.4]{JKMT}}]\label{lm:007} 
A Hoffman graph $\Ho$ is indecomposable 
if and only if 
$U(\mathcal{S}(\Ho))$ is connected. 
\end{lemma}

\begin{lemma}\label{lm:offB3} 
Let $\Ho$ be a fat Hoffman graph with 
$\lambda_{\min}(\Ho) \geq -3$.
Then for any two distinct vertices $x$ and $y$,
$B(\Ho)_{x,y}=\pm1$ if and only if $\{x,y\}
\in E^\pm(\mathcal{S}(\Ho))$.
\end{lemma}
\begin{proof}
It follows from the definition of $B(\Ho)$ and $E^\pm(\mathcal{S}(\Ho))$
that, for any two distinct slim vertices of $\Ho$,
\begin{align*}
\{x,y\}\in E^+(\mathcal{S})&\iff B(\Ho)_{xy}=1,\\
\{x,y\}\in E^-(\mathcal{S})&\iff B(\Ho)_{xy}\leq-1.
\end{align*}
Since $\lambda_{\min}(\Ho) \geq -3$, Lemma~\ref{lm:offB}
applies, and $B(\Ho)_{xy}\leq-1$ forces $B(\Ho)_{xy}=-1$.
\end{proof}

\subsection{Block graphs}

A vertex $v$ in a graph $G$ 
is called a \emph{cut vertex} of $G$ 
if the number of connected components of $G-v$ 
is greater than that of $G$. 
A connected graph $G$ is said to be \emph{$2$-connected} 
if $G$ has no cut vertex. 
A \emph{block} in a graph is a maximal $2$-connected subgraph of the graph. 
Two distinct blocks have at most one vertex in common.
An \emph{end block} is a block having at most one cut vertex. 
We define the \emph{block graph} $\mathbb{B}(G)$ of a graph $G$ 
to be the graph 
whose vertex set is the set of blocks of $G$ 
and two distinct blocks are adjacent in $\mathbb{B}(G)$ 
if and only if they have a common vertex in $G$. 
A \emph{block graph} is a graph isomorphic to the block graph of some graph.

\begin{lemma}[{\cite[Theorems 1 and 2]{Harary}}]\label{lm:H}
A graph $G$ is a block graph if and only if
every block of $G$ is a clique. 
\end{lemma}

\begin{lemma}[{\cite[Proposition 1]{BM86}}]\label{lm:block1}
If a graph $G$ contains neither 
the diamond graph $K_{1,1,2}$ or a cycle of length at least four 
as an induced subgraph, 
then $G$ is a block graph. 
\end{lemma}

A graph is said to be \emph{claw-free} 
if it does not contain $K_{1,3}$ as an induced subgraph. 

\begin{lemma}\label{lm:block3}
If a connected block graph $G$ is claw-free, 
then $\mathbb{B}(G)$ is a tree.
Let $n(B)$ denote the number of non cut vertices of a block $B$ of $G$.
Let $T$ be the tree obtained from $\mathbb{B}(G)$ by attaching
$n(B)$ pendant edges to the vertex $B$, for each vertex $B$ of $\mathbb{B}(G)$.
Then $G$ is isomorphic to the line graph $L(T)$ of $T$.
\end{lemma}

\begin{proof}
Let $B$ be a vertex of $\mathbb{B}(G)$ which is not a leaf.
Then there are two neighbors $B_1, B_2$ of
$B$ in $\mathbb{B}(G)$. Since $G$ is claw-free, there are
distinct vertices $v_1,v_2$ of $B$ such that
$B \cap B_i=\{v_i\}$ for $i=1,2$.
Since $v_i$ is a cut vertex, 
$B-v_i$ and $B_i-v_i$ belong to the
different connected components of $G-v_i$.
If $\mathbb{B}(G)-B$ is connected, then there is a path
in $\mathbb{B}(G)-B$ connecting $B_1$ and $B_2$.
This implies that there is a path in $G-v_2$ connecting $v_1$
and a vertex of $B_2-v_2$. But then $B-v_2$
and $B_2-v_2$ belong to the same connected component
of $G-v_2$, a contradiction. Therefore, every vertex of $\mathbb{B}(G)$ is
either a leaf or a cut vertex. Since $\mathbb{B}(G)$ is connected,
we conclude that $\mathbb{B}(G)$ is a tree.

Since there is a bijection between the set of edges of $\mathbb{B}(G)$
and the set of cut vertices of $G$, the set of edges of $T$
bijectively corresponds to the set of vertices $G$. Then it is
easy to see that this bijective correspondence
between the vertices of $L(T)$ and those of $G$ preserves the adjacency.
\end{proof}

It is well known that, in a block graph, 
there exists a unique shortest path 
between two distinct vertices.

\begin{lemma}\label{lm:block4}
Let $G$ be a claw-free block graph, and let $B$ and $B'$
be blocks of $G$ with $V(B)\cap V(B')=\emptyset$.
If $v$ and $v'$ are vertices in $B$ and $B'$,
respectively,
then 
\[
d_{\mathbb{B}(G)}({B}, {B}') = 
\left\{
\begin{array}{ll}
d_G(v,v')-1 & \text{ if } 
|V(P) \cap (V(B) \cup V(B'))|=4, \\
d_G(v,v') & \text{ if } 
|V(P) \cap (V(B) \cup V(B'))|=3, \\
d_G(v,v')+1 & \text{ if } 
|V(P) \cap (V(B) \cup V(B'))|=2, \\
\end{array}
\right.
\]
where $P$ is the shortest path between $v$ and $v'$ in $G$. 
\end{lemma}

\begin{proof}
Let $P=(v=u_0, u_1, \ldots, u_k=v')$ 
be the shortest path between $v$ and $v'$ in $G$, 
where $k=d_G(v,v')$. 
Let ${B}_{i}$ be the block of $G$ 
containing $\{u_{i-1}, u_{i}\}$ for $i=1, \ldots,k$. 
If 
$|V(P) \cap (V(B) \cup V(B'))|=4$, 
then 
${B}_{1}=B$ 
and ${B}_{k}=B'$. 
Therefore 
$({B}_{1}, \ldots, {B}_{k})$ 
is the shortest path 
between $B$ and $B'$ in $\mathbb{B}(G)$. 
Thus $d_{\mathbb{B}(G)}(B, B') = k-1$. 
If 
$|V(P) \cap (V(B) \cup V(B'))|=3$, 
then either 
${B}_{1} \neq B$ 
and ${B}_{k} = B'$ 
or 
${B}_{1} = B$ 
and ${B}_{k} \neq B'$.
Therefore 
$(B, {B}_{1}, \ldots, {B}_{k})$ 
or 
$({B}_{1}, \ldots, {B}_{k}, B')$ 
is the shortest path 
between $B$ and $B'$ in $\mathbb{B}(G)$. 
Thus $d_{\mathbb{B}(G)}(B, B') = k$. 
If 
$|V(P) \cap (V(B) \cup V(B'))|=2$, 
then 
${B}_{1} \neq B$ 
and ${B}_{k} \neq B'$. 
Therefore 
$(B, {B}_{1}, \ldots, {B}_{k}, B')$ 
is the shortest path 
between $B$ and $B'$ in $\mathbb{B}(G)$. 
Thus $d_{\mathbb{B}(G)}(B, B') = k+1$. 
Hence the lemma holds. 
\end{proof}

\section{Some Hoffman graphs $\Ho$ with $\lambda_{\min}(\Ho)\leq-3$}
\label{sec:-3}

For a positive integer $t$, 
let $\mathfrak{K}_{1,t}$ be 
the connected Hoffman graph having exactly one slim vertex 
and $t$ fat vertices. 
Note that 
and $\lambda_{\min}(\mathfrak{K}_{1,t}) = -t$. 
A Hoffman graph is said to be {\it $\mathfrak{K}_{1,t}$-free} 
if it does not contain $\mathfrak{K}_{1,t}$ 
as an induced Hoffman subgraph. 
If a Hoffman graph $\Ho$ has smallest eigenvalue greater than $-t$, 
then $\Ho$ is $\mathfrak{K}_{1,t}$-free by Lemma~\ref{lm:001}. 
By a Hoffman graph {\it containing $\mathfrak{K}_{1,2}$}, we mean
a Hoffman graph in which 
some slim vertex has two fat neighbors. 

In this section, 
we give some Hoffman graphs $\Ho$ with
$\lambda_{\min}(\Ho)\leq-3$. 
Lemma~\ref{lm:offB3} implies that unless all
the off-diagonal entries of $B(\Ho)$ are in $\{0,\pm1\}$,
we have $\lambda_{\min}(\Ho)<-3$. 
Thus, in the proofs
of lemmas in this section, we may assume without
loss of generality that all 
the off-diagonal entries of $B(\Ho)$ are in $\{0,\pm1\}$.

The graphs obtained in this section will form a set
of forbidden subgraphs for Hoffman graphs with smallest
eigenvalue greater than $-3$.
This set of forbidden graphs will be used in the next section
to determine the structure of 
the special graph $\mathcal{S}(\Ho)$ of a
fat Hoffman graph $\Ho$ containing $\mathfrak{K}_{1,2}$ 
with smallest eigenvalue greater than $-3$. 

\begin{lemma}\label{lm:104}
Let $\Ho$ be a fat Hoffman graph. 
If $U({\mathcal{S}}(\Ho))$ is a path 
each of whose two end vertices 
has two fat neighbors in $\Ho$, 
then $ \lambda_{\min}(\Ho)\le -3 $.
\end{lemma}

\begin{proof}
By the assumption, we have
\[ 
B(\Ho)=
\begin{pmatrix}
-2&\epsilon_1& 0 & \cdots & \cdots & 0 \\
\epsilon_1&\delta_2&\epsilon_2 & \ddots & & \vdots \\
0&\epsilon_2&\delta_3&\ddots & \ddots & \vdots \\
\vdots & \ddots&\ddots&\ddots&\epsilon_{n-2} & 0 \\
\vdots & & \ddots & \epsilon_{n-2}&\delta_{n-1}&\epsilon_{n-1}\\
0 & \cdots & \cdots & 0 & \epsilon_{n-1} & -2\\
\end{pmatrix},
\] 
where $n$ is the number of slim vertices of $\Ho$, and
$\epsilon_1,\dots,\epsilon_{n-1}\in\{\pm1\}$,
$\delta_2,\dots,\delta_{n-1}\leq-1$.
Let $B'$ be the matrix obtained from $B(\Ho)$ by replacing
$\delta_i$ by $-1$ for all $i=2,\dots,n-1$. As
$\lambda_{\min}(\Ho)\leq\lambda_{\min}(B')$, it suffices
to show $\lambda_{\min}(B')=-3$.
Multiplying $B'$ by the diagonal
matrix whose diagonal entries are
$1,\epsilon_1,\epsilon_1\epsilon_2,\dots,\prod_{i=1}^{n-1}\epsilon_i$
from both sides, 
we see that $B'$ is similar to the matrix
\[
\begin{pmatrix}
-2&   1&      &      &   \\
1&     -1&     1&      &   \\
 &\ddots&\ddots&\ddots&   \\
 &      &     1&     -1&1  \\
 &      &      &     1&-2
\end{pmatrix}.
\]
The smallest eigenvalue of this matrix is $-3$
by \cite[Theorem 3]{Losonczi}.
This implies $ \lambda_{\min}(B')=-3 $.
\end{proof}

\begin{lemma}\label{lm:104c}
Let $\Ho$ be a fat indecomposable Hoffman graph. 
If $\Ho$ contains at least two $\mathfrak{K}_{1,2}$, 
then $ \lambda_{\min}(\Ho) \leq -3 $. 
\end{lemma}

\begin{proof}
Let $v_1$ and $v_2$ be the slim vertices of the two 
$\mathfrak{K}_{1,2}$. 
If $v_1=v_2$, then $\Ho$ contains 
$\mathfrak{K}_{1,3}$ and thus $ \lambda_{\min}(\Ho) \leq -3$. 
Now we assume that $v_1 \neq v_2$. 
Since $\Ho$ is indecomposable, 
$\mathcal{S}(\Ho)$ is connected by Lemma~\ref{lm:007}.
Let $\mathcal{P}$ be a shortest path from $v_1$ to $v_2$ 
in $\mathcal{S}(\Ho)$. 
Let $\Ho_{\mathcal{P}}$ be the Hoffman subgraph 
of $\Ho$ induced by the slim vertices which belong to $\mathcal{P}$ 
and their fat neighbors. 
Then $\mathcal{S}(\Ho_{\mathcal{P}}) = \mathcal{P}$. 
By Lemmas~\ref{lm:001} and \ref{lm:104}, 
$\lambda_{\min}(\Ho) \leq 
\lambda_{\min}(\Ho_{\mathcal{P}}) \leq -3$. 
\end{proof}

\begin{lemma}\label{lm:102-2}
Let $\Ho$ be a fat Hoffman graph containing $\mathfrak{K}_{1,2}$. 
Suppose that the slim vertex $v^*$ in $\mathfrak{K}_{1,2}$
has two slim neighbors which are not adjacent in $\Ho$.
Then $\lambda_{\min}(\Ho)\leq-3$. 
\end{lemma}

\begin{proof}
Let $u$ and $w$ be slim neighbors of $v$ which are not adjacent in 
$\mathcal{S}(\Ho)$.  
Then the matrix $B(\Ho')$ of the Hoffman subgraph $\Ho'$ of $\Ho$ 
induced by $u,v,w$ and their fat neighbors is 
one of the matrices 
\[ 
\begin{pmatrix}
-1 & 1 & 0 \\
1 & -2 & 1 \\
0 & 1 & -1 
\end{pmatrix}, 
\begin{pmatrix}
-1 & -1 & 0 \\
-1 & -2 & 1 \\
0 & 1 & -1 
\end{pmatrix}, 
\begin{pmatrix}
-1 & -1 & 0 \\
-1 & -2 & -1 \\
0 & -1 & -1 \\
\end{pmatrix}, 
\]
which 
have smallest eigenvalue $-3$ 
and thus $ \lambda_{\min}(\Ho) \leq -3 $ 
by Lemma~\ref{lm:001}. 
\end{proof}

\begin{lemma}\label{lm:Dn}
Let $\Ho$ be a fat Hoffman graph containing $\mathfrak{K}_{1,2}$. 
Let $\mathsf{D}_n$ $(n \geq 4)$ be the graph defined by 
$V(\mathsf{D}_n)=\{v_1,v_2,v_3, \ldots, v_n\}$ 
and $E(\mathsf{D}_n)= 
\{ \{v_1,v_2\}$, $\{v_2,v_3\}, \ldots, \{v_{n-2},v_{n-1}\}, \{v_{n-2}, v_n\}\}$. 
If $U({\mathcal{S}}(\Ho)) = \mathsf{D}_n$ $(n \geq 4)$ 
and $v_1$ has two fat neighbors in $\Ho$, 
then $ \lambda_{\min}(\Ho) \leq -3 $.
\end{lemma}

\begin{proof}
By assumption, we have 
\[ 
B(\Ho)=
\begin{pmatrix}
-2         & \epsilon_1 & 0          & \cdots & \cdots         & 0 & 0 \\
\epsilon_1 & -1         & \epsilon_2 & \ddots &                & \vdots & \vdots \\
0          & \epsilon_2 & -1         & \ddots & \ddots         & \vdots & \vdots \\
\vdots     & \ddots     & \ddots     & \ddots & \epsilon_{n-3} & 0      & 0 \\
\vdots     & & \ddots & \epsilon_{n-3} & -1      & \epsilon_{n-2} & \epsilon_{n-1}  \\
0      & \cdots & \cdots & 0 & \epsilon_{n-2} & -1 & 0 \\
0      & \cdots & \cdots & 0 & \epsilon_{n-1} & 0 & -1 \\
\end{pmatrix},
\]
where $\epsilon_1,\dots,\epsilon_{n-1}\in\{\pm1\}$.
Let $\mathbf{x} \in \mathbb{R}^n$ be a vector defined by 
\[
(\mathbf{x})_i=
\begin{cases}
2 \prod_{k=i}^{n-3} (-\epsilon_k) & \text{if } 1 \leq i \leq n-3,\\ 
2 & \text{if } i= n-2, \\
- \epsilon_{i-1} & \text{if } i= n-1, n. 
\end{cases}
\]
Then $ (B(\Ho)+3I)\mathbf{x}={\bf 0} $.
Hence $ B(\Ho) $ has $ -3 $ as an eigenvalue.
This implies $ \lambda_{\min}(\Ho)\le -3 $. 
\end{proof}

\begin{lemma}\label{lm:201}
Let $\Ho$ be a Hoffman graph. 
Let $\mathcal{T}$ be a triangle 
in the special graph $\mathcal{S}(\Ho)$ 
such that every vertex in $\mathcal{T}$ has 
exactly one fat neighbor in $\Ho$. 
If $\mathcal{T}$ has a $(-)$-edge, 
then $ \lambda_{\min}(\Ho) \leq -3 $. 
\end{lemma}

\begin{proof}
By \cite[Lemma 3.11]{MST1}, 
the number of $(-)$-edges in $\mathcal{T}$ 
cannot be two. 
Therefore, 
the number of $(-)$-edges in $\mathcal{T}$ 
is one or three. 
Then, 
the matrix $B(\Ho_{\mathcal{T}})$ of 
the Hoffman subgraph $\Ho_{\mathcal{T}}$ induced 
by ${\mathcal{T}}$ and their fat neighbors 
is one of the matrices 
\[ 
\begin{pmatrix}
-1 & 1 & 1  \\
1 & -1  & -1 \\
1 & -1  & -1  \\
\end{pmatrix}, 
\begin{pmatrix}
-1 & -1 & -1 \\
-1 & -1 & -1 \\
-1 & -1 & -1 \\
\end{pmatrix} 
\]
which have smallest eigenvalue $-3$. 
By Lemma~\ref{lm:001}, 
we obtain $\lambda_{\min}(\Ho) \leq 
\lambda_{\min}(\Ho_{\mathcal{T}}) = -3 $. 
\end{proof}

\begin{lemma}\label{lm:201b}
Let $\Ho$ be a Hoffman graph. 
Let $\mathcal{T}$ be a triangle 
in the special graph $\mathcal{S}(\Ho)$ 
such that two vertices $v_1$, $v_2$ in $\mathcal{T}$ 
have exactly one fat neighbor in $\Ho$ 
and the other vertex $v^*$ in $\mathcal{T}$ 
has exactly two fat neighbors in $\Ho$. 
If $\mathcal{T}$ has a $(-)$-edge, 
then 
$\lambda_{\min}(\Ho) \leq -3$ 
or 
$E^+(\mathcal{T})=\{ \{ v_1, v_2\} \}$ and 
$E^-(\mathcal{T})=\{ \{ v^*, v_1\}, \{ v^*, v_2\} \}$. 
\end{lemma}

\begin{proof}
Suppose that $E^{\pm}(\mathcal{T})$ 
is different from the one described.
It is enough to consider the following cases: 
\begin{itemize}
\item[(a)]
$E^+(\mathcal{T})=\{ \{ v^*, v_1\}, \{ v^*, v_2\}\}$, 
$E^-(\mathcal{T})=\{ \{ v_1, v_2\} \}$, 
\item[(b)]
$E^+(\mathcal{T})=\{ \{ v^*, v_1\}, \{ v_1, v_2\} \}$, 
$E^-(\mathcal{T})=\{ \{ v^*, v_2\} \}$, 
\item[(c)]
$E^+(\mathcal{T})=\{ \{ v^*, v_1\} \}$, 
$E^-(\mathcal{T})=\{ \{ v^*, v_2\}, \{ v_1, v_2\} \}$, 
\item[(d)]
$E^+(\mathcal{T})=\emptyset$, 
$E^-(\mathcal{T})=\{ \{ v^*, v_1\}, \{ v^*, v_2\}, \{ v_1, v_2\} \}$, 
\end{itemize}
First, consider the case (c). 
Since $\{ v^*, v_2\}$ and $\{ v_1, v_2\}$ 
are $(-)$-edges in $\mathcal{S}(\Ho)$, 
$v^*$ and $v_2$ have a common fat neighbor, say $f$, in $\Ho$, 
and 
$v_1$ and $v_2$ have a common fat neighbor, say $f'$, in $\Ho$. 
Since $v_2$ has exactly one fat neighbor in $\Ho$, $f=f'$. 
Then $f$ is a common fat neighbor of $v^*$ and $v_1$, 
which is a contradiction to the fact that 
$\{ v^*, v_1\}$ 
is a $(+)$-edge in $\mathcal{S}(\Ho)$. 

In the cases (a), (b), and (d), 
the matrices $B(\Ho_{\mathcal{T}})$, 
where $\Ho_{\mathcal{T}}$ denotes 
the Hoffman subgraph of $\Ho$ induced 
by ${\mathcal{T}}$ and their fat neighbors, 
are the following matrices, respectively, 
\[ 
\begin{pmatrix}
-2 & 1 & 1  \\
1 & -1  & -1 \\
1 & -1  & -1  \\
\end{pmatrix}, 
\begin{pmatrix}
-2 & 1 & -1 \\
1 & -1 & 1 \\
-1 & 1 & -1 \\
\end{pmatrix}, 
\begin{pmatrix}
-2 & -1 & -1 \\
-1 & -1 & -1 \\
-1 & -1 & -1 \\
\end{pmatrix} 
\]
which have smallest eigenvalue less than $-3$. 
Then, by Lemma~\ref{lm:001}, we have 
$\lambda_{\min}(\Ho) \leq \lambda_{\min}(\Ho_{\mathcal{T}}) < -3 $.
Hence the lemma holds. 
\end{proof}

\begin{lemma}\label{lm:K112}
Let $\Ho$ be a fat Hoffman graph containing $\mathfrak{K}_{1,2}$. 
If $U({\mathcal{S}}(\Ho)) = K_{1,1,2}$, 
then $ \lambda_{\min}(\Ho) \leq -3 $.
\end{lemma}

\begin{proof}
Let $v^*$ be a slim vertex which has two fat neighbors in $\Ho$. 
If $v^*$ is a vertex of degree three in $U({\mathcal{S}}(\Ho)) = K_{1,1,2}$, 
then it follows from Lemma~\ref{lm:102-2} 
that $ \lambda_{\min}(\Ho) \leq -3 $. 
Therefore, we assume that $v^*$ is a vertex of degree two 
in $U({\mathcal{S}}(\Ho)) = K_{1,1,2}$. 
Then, 
\[ 
B(\Ho) = 
\begin{pmatrix}
-2 & \epsilon_{12} & \epsilon_{13} & 0 \\
\epsilon_{12} & -1 & \epsilon_{23} & \epsilon_{24} \\
\epsilon_{13} & \epsilon_{23} & -1 & \epsilon_{34} \\
0 & \epsilon_{24} & \epsilon_{34} & -1 \\
\end{pmatrix}, 
\]
where $\epsilon_{12}, \epsilon_{13}, \epsilon_{23}, 
\epsilon_{24}, \epsilon_{34} \in \{\pm 1\}$. 
By using computer, we can check that 
$B(\Ho)$ has the smallest eigenvalue at most $-3$ for any case. 
Hence $ \lambda_{\min}(\Ho) \leq -3 $. 
\end{proof}

\begin{lemma}\label{lm:103a}
Let $\Ho$ be a fat Hoffman graph containing $\mathfrak{K}_{1,2}$. 
Let $\mathcal{P}$ and $\mathcal{K}$ denote edge-signed graphs 
such that $U(\mathcal{P})$ is a path and $U(\mathcal{K}) \cong K_{1,1,2}$, 
respectively. 
If $\mathcal{S}(\Ho)$ is 
the graph obtained from $\mathcal{P}$ and $\mathcal{K}$ 
by identifying an end vertex of $\mathcal{P}$ and a vertex 
in $\mathcal{K}$, and 
if another end vertex of $\mathcal{P}$ has two fat neighbors 
in $\Ho$, 
then $\lambda_{\min}(\Ho) \leq -3 $. 
\end{lemma}

\begin{proof}
Since $\mathcal{S}(\Ho)$ is connected, 
$\Ho$ is indecomposable by Lemma~\ref{lm:007}. 
If $\Ho$ contains at least two $\mathfrak{K}_{1,2}$, 
then it follows from Lemma~\ref{lm:104c} 
that $ \lambda_{\min}(\Ho) \leq -3 $. 
Therefore, we assume that $\Ho$ contains exactly one $\mathfrak{K}_{1,2}$. 
Since 
the end vertex of $\mathcal{P}$ that is not in 
$\mathcal{K}$
has two fat neighbors in $\Ho$, 
every vertex in 
$\mathcal{K}$ 
has exactly one fat neighbor in $\Ho$. 
If 
$\mathcal{K}$ 
has a $(-)$-edge, 
it follows from Lemma~\ref{lm:201} that $ \lambda_{\min}(\Ho) \leq -3 $. 
Thus, now we assume that 
all the edges in 
$\mathcal{K}$ 
are $(+)$-edges. 
Let $n$ be the number of vertices in $\mathcal{P}$. Then
\[ 
B(\Ho)=
\begin{pmatrix}
-2 & \epsilon_1 & 0 & \cdots & \cdots & 0 & 0 & 0 & 0 \\
\epsilon_1 & -1 & \epsilon_2 & \ddots & & \vdots & \vdots & \vdots & \vdots  \\
0 & \epsilon_2 & -1 & \ddots & \ddots & \vdots& & &  \\
\vdots & \ddots & \ddots & \ddots & \epsilon_{n-2} & 0 &  \vdots & \vdots  & \vdots   \\
\vdots & & \ddots & \epsilon_{n-2} & -1 & \epsilon_{n-1} & 0& 0& 0 \\
0 & \cdots & \cdots & 0 & \epsilon_{n-1} & -1 & 1 & 1 & \epsilon \\
0 & \cdots & & \cdots & 0& 1 & -1 & \delta & 1 \\
0 & \cdots & & \cdots & 0& 1 & \delta & -1 & 1 \\
0 & \cdots & & \cdots & 0& \epsilon  & 1 & 1 & -1 \\
\end{pmatrix},
\]
where $\epsilon_1, \dots, \epsilon_{n-1} \in \{ \pm1 \}$ and
$\{ \epsilon, \delta\} = \{ 0,1 \}$. 
Let $\mathbf{x} \in \mathbb{R}^{n+3}$ be a vector defined by 
\[
	(\mathbf{x})_i=\begin{cases}
	2\prod_{k=i}^{n-1}(- \epsilon_k) & \text{if } 1 \leq i \leq n-1,\\
	2 & \text{if } i=n,\\
	-1 & \text{if } i=n+1,n+2,\\
	2(\epsilon-1)/(\epsilon-2)&\text{if } i=n+3. 
	\end{cases}
\]
Since $ \epsilon +\delta=1,\ \epsilon\delta=0,\ \epsilon^2-\epsilon=0 $, 
we obtain $ (B(\Ho)+3I)\mathbf{x}={\bf 0} $.
Hence $ B(\Ho) $ has $ -3 $ as an eigenvalue.
This implies $ \lambda_{\min}(\Ho)\le -3 $.
\end{proof}

\begin{lemma}\label{lm:-2c}
Let $\Ho$ be a Hoffman graph. 
Let $\mathcal{C}$ be a cycle of length at least four 
in the special graph $\mathcal{S}(\Ho)$ 
such that every vertex in $\mathcal{C}$ has exactly one fat neighbor in $\Ho$. 
If the number of $(+)$-edges in $\mathcal{C}$ 
is even, 
then $ \lambda_{\min}(\Ho)\le -3 $. 
\end{lemma}

\begin{proof}
Let $n$ be the length of the cycle $\mathcal{C}$. 
Let $\Ho'$ be the Hoffman subgraph of $\Ho$ 
induced by $\mathcal{C}$ and their fat neighbors. 
Then 
\[
B(\Ho') =
\begin{pmatrix}
-1 & \epsilon_1 & 0 & \cdots & 0 & \epsilon_n \\
\epsilon_1 & -1 & \epsilon_2 & \ddots &  & 0 \\
0 & \epsilon_2 & -1 & \ddots & \ddots & \vdots  \\
\vdots & \ddots & \ddots & \ddots & \epsilon_{n-2} & 0 \\
0 & & \ddots & \epsilon_{n-2} & -1 & \epsilon_{n-1}\\
\epsilon_n & 0 & \cdots & 0 & \epsilon_{n-1} & -1 
\end{pmatrix}, 
\]
where $\epsilon_1,\dots,\epsilon_n\in\{\pm1\}$. 
Since the number of $(+)$-edges in $\mathcal{C}$ 
is even, we have 
\[
\prod_{k=1}^n\epsilon_k=(-1)^n.
\]
Let $\mathbf{x} \in \mathbb{R}^n$ be the vector defined by 
\[
(\mathbf{x})_i = 
\left\{
\begin{array}{ll}
1 & i=1, \\
\prod_{k=1}^{i-1} (-\epsilon_k) & 2 \leq i \leq n. 
\end{array}
\right.
\]
Then $(B(\Ho')+3I) \mathbf{x}={\bf 0}$. 
Thus $\Ho'$ has $-3$ as an eigenvalue. 
By Lemma~\ref{lm:001}, 
$ \lambda_{\min}(\Ho) \leq \lambda_{\min}(\Ho') \leq -3 $. 
\end{proof}

\begin{remark}
In the proof of Lemma~\ref{lm:-2c}, multiplying $B(\Ho')$ by the diagonal
matrix whose diagonal entries are
$1,\epsilon_1,\epsilon_1\epsilon_2,\dots,\prod_{i=1}^{n-1}\epsilon_i$
from both sides, 
the resulting matrix is circulant
with the first row $(-1,1,0,\dots,0,1)$ when $n$ is even.
Thus $B(\Ho')$ has smallest eigenvalue $-3$.
If $n$ is odd, it contains the submatrix
\[
\begin{pmatrix}
-1&1&-1\\
1&-1&1\\
-1&1&-1
\end{pmatrix}
\]
which has the smallest eigenvalue $-3$ as we have already seen in
the proof of Lemma~\ref{lm:201}.
Thus $B(\Ho')$ has smallest eigenvalue at most $-3$.
\end{remark}

Let $P_m$ denote the path of length $m-1$ consisting of $m$
vertices, and let $C_n$ denote the cycle of length $n$.

\begin{lemma}\label{lm:-3c}
Let $\Ho$ be a fat Hoffman graph containing $\mathfrak{K}_{1,2}$. 
Let $\mathcal{P}$ and $\mathcal{C}$ denote edge-signed graphs 
such that $U(\mathcal{P}) \cong P_m$ and $U(\mathcal{C}) \cong C_n$, 
respectively. 
Suppose that $\mathcal{S}(\Ho)$ is 
the graph obtained from $\mathcal{P}$ and $\mathcal{C}$ 
by adding two edges between an end vertex $v_m$ of $\mathcal{P}$ 
and two adjacent vertices $v_{m+1}$ and $v_{m+2}$ of $\mathcal{C}$, 
and that another end vertex $v_1$ of $\mathcal{P}$ has 
two fat neighbors in $\Ho$. 
Then $\lambda_{\min}(\Ho) \leq -3 $. 
\end{lemma}

\begin{proof}
If the number of $(+)$-edges in $\mathcal{C}$ is even, 
then we have $\lambda_{\min}(\Ho) \leq -3 $ by Lemma~\ref{lm:-2c}. 
Therefore, we assume that 
the number of $(+)$-edges in $\mathcal{C}$ is odd. 
Since $\mathcal{S}(\Ho)$ is connected, 
$\Ho$ is indecomposable by Lemma~\ref{lm:007}. 
If $\Ho$ contains at least two $\mathfrak{K}_{1,2}$, 
then it follows from Lemma~\ref{lm:104c} 
that $ \lambda_{\min}(\Ho) \leq -3 $. 
Therefore, we assume that $\Ho$ contains exactly one $\mathfrak{K}_{1,2}$. 
Since 
the end vertex $v_1$ of $\mathcal{P}$ 
has two fat neighbors in $\Ho$, 
every vertex in 
$\mathcal{S} - \{v_1\}$
has exactly one fat neighbor in $\Ho$. 
If the triangle $\{v_{n}, v_{n+1}, v_{n+2} \}$ 
has a $(-)$-edge, 
it follows from Lemma~\ref{lm:201} that $ \lambda_{\min}(\Ho) \leq -3 $. 
Thus, now we assume that 
all the edges in the triangle $\{v_{n}, v_{n+1}, v_{n+2} \}$ 
are $(+)$-edges. 
Then $B(\Ho)$ is the matrix 
\[ 
\left(\begin{array}{cccccccccccc}
-2 & \delta_1 & 0 & \cdots & \cdots & 0 & 0 & 0 & \cdots & & \cdots & 0 \\
\delta_1 & -1 & \delta_2 & \ddots & & \vdots & \vdots & \vdots & & & & \vdots  \\
0 & \delta_2 & -1 & \ddots & \ddots & \vdots  &&&&&& \\
\vdots & \ddots & \ddots & \ddots & \delta_{m-2} & 0 & \vdots & \vdots &&&& \vdots \\
\vdots & & \ddots & \delta_{m-2} & -1 & \delta_{m-1} & 0 & 0 & \cdots & & \cdots & 0\\
0 & \cdots & \cdots & 0 & \delta_{m-1} & -1 & 1 & 1 & 0 & \cdots & \cdots & 0 \\
0 & \cdots && \cdots & 0 & 1& -1 & 1 & 0 & \cdots & 0 & \epsilon_n \\
0 & \cdots && \cdots & 0 & 1& 1 & -1 & \epsilon_2 & \ddots & & 0  \\
\vdots & & & & \vdots & 0 & 0 & \epsilon_2 & -1 & \ddots & \ddots & \vdots  \\
& & & & & \vdots & \vdots & \ddots & \ddots & \ddots & \epsilon_{n-2} & 0 \\
\vdots &&&& \vdots & \vdots &  0 & & \ddots & \epsilon_{n-2} & -1 & \epsilon_{n-1}\\
0 & \cdots & & \cdots & 0 & 0 & \epsilon_n & 0 & \cdots & 0 & \epsilon_{n-1} & -1
\end{array}\right),
\]
where 
$\delta_1,\dots,\delta_{m-1}, \epsilon_2,\dots,\epsilon_n\in\{\pm1\}$. 
Since the number of $(+)$-edges in $\mathcal{C}_n$ is odd, 
we have 
\[
\prod_{k=2}^n\epsilon_k=(-1)^{n-1}.
\]
Let $\mathbf{x} \in \mathbb{R}^{m+n}$ be the vector defined by 
\[
(\mathbf{x})_i = 
\left\{
\begin{array}{ll}
2 & \text{if } i=1 \\ 
2 \cdot \prod_{k=1}^{i-1}(-\delta_k) & \text{if } 2 \leq i \leq m \\
- \prod_{k=1}^{m-1}(-\delta_k) & \text{if } m+1\leq i\leq m+2 \\
-\prod_{k=1}^{m-1}(-\delta_k) \cdot 
\prod_{k=2}^{i-m-1}(-\epsilon_k) & \text{if } m+3 \leq i \leq m+n 
\end{array} 
\right. 
\]
Then $(B(\Ho)+3I) \mathbf{x}={\bf 0}$. 
Therefore, $\Ho$ has $-3$ as an eigenvalue 
and thus 
$ \lambda_{\min}(\Ho) \leq -3 $. 
\end{proof}

\begin{lemma}\label{lm:103A}
Let $\Ho$ be a fat indecomposable Hoffman graph 
containing $\mathfrak{K}_{1,2}$. 
If $U({\mathcal{S}}(\Ho))$ contains $K_{1,1,2}$ as an induced subgraph,  
then $\lambda_{\min}(\Ho) \leq -3$. 
\end{lemma}

\begin{proof}
If $\Ho$ contains at least two $\mathfrak{K}_{1,2}$, 
then $\lambda_{\min}(\Ho) \leq -3$ by Lemma~\ref{lm:104c}. 
So we assume that $\Ho$ contains exactly one $\mathfrak{K}_{1,2}$. 
Let $v^*$ be the slim vertex of $\mathfrak{K}_{1,2}$. 
Suppose that $U(\mathcal{S}(\Ho))$ contains $K_{1,1,2}$. 
Let $K$ be an induced subgraph of $U(\mathcal{S}(\Ho))$ 
such that $K \cong K_{1,1,2}$. 
If $v^*$ is in $K$, then it follows from Lemma~\ref{lm:K112} 
that $\lambda_{\min}(\Ho) \leq -3$. 
Consider the case where $v^*$ is not in $K$. 
Note that $\mathcal{S}(\Ho)$ is conneted 
since $\Ho$ is indecomposable. 
Let $P$ be a shortest path in $U(\mathcal{S}(\Ho))$ 
from $v^*$ to a vertex in $K$. 
Let $\Ho'$ be the Hoffman subgraph of $\Ho$ 
induced by the slim vertices which belong to 
$P$ or $K$ and their fat neighbors. 
Then $U(\mathcal{S}(\Ho')) = P \cup K$, 
where the end vertex of $P$ other than $v^*$ 
is identified with a vertex in $K$. 
By Lemma~\ref{lm:103a}, 
$\lambda_{\min}(\Ho') \leq -3$. 
Since $\Ho'$ is an induced Hoffman subgraph of $\Ho$, 
$\lambda_{\min}(\Ho) \leq \lambda_{\min}(\Ho')$  
by Lemma~\ref{lm:001}. 
Therefore $\lambda_{\min}(\Ho) \leq -3$. 
\end{proof}

\begin{lemma}\label{lm:103B}
Let $\Ho$ be a fat indecomposable Hoffman graph 
containing $\mathfrak{K}_{1,2}$. 
If $U({\mathcal{S}}(\Ho))$ contains 
$C_n$ $(n \geq 4)$ as an induced subgraph,  
then $ \lambda_{\min}(\Ho) \leq -3 $. 
\end{lemma}

\begin{proof}
Let $C$ be an induced subgraph of $U(\mathcal{S}(\Ho))$ 
such that $C \cong C_{n}$ for some $n \geq 4$. 
If $\Ho$ contains at least two $\mathfrak{K}_{1,2}$, 
then $\lambda_{\min}(\Ho) \leq -3$ by Lemma~\ref{lm:104c}. 
So we may assume that $\Ho$ contains exactly one $\mathfrak{K}_{1,2}$. 
Let $v^*$ be the slim vertex of $\mathfrak{K}_{1,2}$. 
If $v^*$ is in $C$, then 
the two slim neighbors of $v^*$ in $C$ are not adjacent,
and the lemma follows from Lemma~\ref{lm:102-2}. 

Consider the case where $v^*$ is not in $C$. 
Note that $\mathcal{S}(\Ho)$ is connected 
since $\Ho$ is indecomposable. 
Let $P$ be a shortest path in $U(\mathcal{S}(\Ho))$ 
from $v^*$ to a vertex in $C$. 
Let $m$ be the number of the vertices of the path $P$. 
Note that $m \geq 2$. 
Consider the subgraph $G$ of $U(\mathcal{S}(\Ho))$ 
induced by the vertices in $P \cup C$. 
Then it follows from the way of taking $P$ and $C$ 
that $G$ is the graph defined by 
$V(G)=\{v_1, \ldots, v_m, v_{m+1}, \ldots, v_{m+n} \}$ 
and $E(G)=\{ \{v_i, v_{i+1} \} \mid 1 \leq i \leq n-1 \} 
\cup \{\{v_{m+1}, v_{m+n}\}\} \cup F$, 
where $F \subseteq \{ \{v_m, v_{m+j} \} \mid 2 \leq j \leq n \}$. 
If $F= \emptyset$, then 
the subgraph of $U(\mathcal{S}(\Ho))$ 
induced by $\{v_1, \ldots, v_m, v_{m+1}, v_{m+2}, v_{m+n} \}$ 
is isomorphic to $\mathsf{D}_{m+3}$. 
If $\{v_m, v_{m+j} \} \in F$ for some $3 \leq j \leq n-1$, 
then the subgraph of $U(\mathcal{S}(\Ho))$ 
induced by $\{v_1, \ldots, v_m, v_{m+1}, v_{m+j} \}$ 
is isomorphic to $\mathsf{D}_{m+2}$. 
If $\{ \{v_m, v_{m+2} \}, \{v_m, v_{m+n} \} \} \subseteq F$, 
then the subgraph of $U(\mathcal{S}(\Ho))$ 
induced by $\{v_1, \ldots, v_m, v_{m+2}, v_{m+n} \}$ 
is isomorphic to $\mathsf{D}_{m+2}$. 
In these three cases, 
it follows from Lemmas~\ref{lm:001} and \ref{lm:Dn} 
that $\lambda_{\min}(\Ho) \leq -3$. 
In the case where 
$F=\{ \{v_m, v_{m+2} \} \}$ or $F=\{ \{v_m, v_{m+n} \} \}$, 
it follows from Lemmas~\ref{lm:001} and \ref{lm:-3c} 
that $\lambda_{\min}(\Ho) \leq -3$. 
Hence $\lambda_{\min}(\Ho) \leq -3$. 
\end{proof}

\begin{lemma}\label{lm:107}
Let $\Ho$ be a fat Hoffman graph containing $\mathfrak{K}_{1,2}$. 
If $U(\mathcal{S}(\Ho))$ contains a claw $K_{1,3}$, 
then $ \lambda_{\min}(\Ho)\le -3 $.
\end{lemma}

\begin{proof}
Let $v^*$ be the vertex which has two fat neighbors in $\Ho$. 
Let $K$ be an induced subgraph of $U(\mathcal{S}(\Ho))$ 
such that $K \cong K_{1,3}$. 
If $v^* \in V(K)$, then 
let $\Ho'$ be the Hoffman subgraph of $\Ho$ 
induced by $V(K)$ and their fat neighbors in $\Ho$. 
Then 
\[ 
B(\Ho') =
\begin{pmatrix}
-2 & \epsilon_1 & \epsilon_2 & \epsilon_3 \\
\epsilon_1& -1  & 0 & 0   \\
\epsilon_2& 0  & -1 & 0   \\
\epsilon_3& 0  & 0 & -1   \\
\end{pmatrix}, 
\quad
\begin{pmatrix}
-1 & \epsilon_1 & \epsilon_2 & \epsilon_3 \\
\epsilon_1& -2  & 0 & 0   \\
\epsilon_2& 0  & -1 & 0   \\
\epsilon_3& 0  & 0 & -1   \\
\end{pmatrix}
\]
where $\epsilon_1, \epsilon_2, \epsilon_3 \in \{\pm 1\}$. 
Therefore, 
we obtain $\lambda_{\min}(\Ho') \leq -3 $. 
If $v^* \not\in V(K)$, then 
consider a shortest path $P$ in $U(\mathcal{S}(\Ho))$ 
from $v^*$ to a vertex in $K$. 
Let $n$ be the number of vertices of $P$. 
Let $G$ be the subgraph of $U(\mathcal{S}(\Ho))$ 
induced by 
$V(P) \cup V(K)$. 
Then we can easily verify that 
$G$ contains an induced subgraph $D$ 
isomorphic to 
$\mathsf{D}_{n+2}$, $\mathsf{D}_{n+3}$, 
or $\mathsf{D}_{n+4}$. 
Let $\Ho'$ be the Hoffman subgraph of $\Ho$ 
induced by $V(D)$. 
By Lemmas~\ref{lm:001} and \ref{lm:Dn}, 
we obtain 
$\lambda_{\min}(\Ho) \leq \lambda_{\min}(\Ho') \leq -3 $. 
\end{proof}

\begin{prop}\label{prop:s3}
Let $\Ho$ be a fat indecomposable Hoffman graph containing
$\mathfrak{K}_{1,2}$. 
If $U(\mathcal{S}(\Ho))$ contains 
$C_n$ $(n\geq4)$, $K_{1,1,2}$, or $K_{1,3}$ as an induced subgraph, 
then $\lambda_{\min}(\Ho) \leq -3$. 
\end{prop}
\begin{proof}
This follows from Lemmas~\ref{lm:103B}, \ref{lm:103A}, and \ref{lm:107}.
\end{proof}

\section{Main result}\label{sec:main}

\begin{lemma}\label{lm:103b}
Let $\Ho$ be a fat indecomposable Hoffman graph 
containing $\mathfrak{K}_{1,2}$ 
with $ \lambda_{\min}(\Ho) > -3 $. 
Then, the slim vertex $v^*$ in $\mathfrak{K}_{1,2}$ 
is not a cut vertex of $U({\mathcal{S}}(\Ho))$.
\end{lemma}

\begin{proof}
Suppose that the slim vertex $v^*$ in $\mathfrak{K}_{1,2}$ 
is a cut vertex of $U({\mathcal{S}}(\Ho))$. 
Let $v_1$ and $v_2$ be neighbors of $v^*$ in ${\mathcal{S}}(\Ho)$ 
such that $v_1$ and $v_2$ belong to different connected components 
in $U({\mathcal{S}}(\Ho)) - v^*$. 
By Lemma~\ref{lm:104c}, each of 
$v_1$ and $v_2$ has only one fat neighbor in $\Ho$. 
Let $\Ho'$ be the Hoffman subgraph of $\Ho$ 
induced by $v^*$, $v_1$, $v_2$ and their fat neighbors. 
Then by Lemma~\ref{lm:offB3},
\[ 
B(\Ho') =
\begin{pmatrix}
-1 & \epsilon_1 & 0 \\
\epsilon_1  & -2 & \epsilon_2 \\
0 & \epsilon_2 & -1 \\
\end{pmatrix}, 
\]
where $\epsilon_1, \epsilon_2 \in \{\pm 1\}$. 
Then, 
we obtain $\lambda_{\min}(\Ho') = -3 $ 
for any cases. 
Since $\Ho'$ is an induced Hoffman subgraph of $\Ho$, 
$\lambda_{\min}(\Ho) \leq \lambda_{\min}(\Ho')$. 
Therefore 
$\lambda_{\min}(\Ho) \leq -3$, 
which is a contradiction. 
Hence the lemma holds. 
\end{proof}

We denote 
the $(+)$-complete graph on $n$ vertices 
by $\mathcal{K}^{+}_n$ and 
the $(-)$-complete graph on $2$ vertices 
by $\mathcal{K}^{-}_2$. 
Let $\mathcal{T}_1^*$ be 
the triangle defined by 
$V(\mathcal{T}_1^*)=\{ v^*, v_1, v_2 \}$, 
$E^+(\mathcal{T}_1^*)=\{ \{ v_1, v_2\} \}$, and 
$E^-(\mathcal{T}_1^*)=\{ \{ v^*, v_1\}, \{ v^*, v_2\} \}$. 

Let $\mathcal{S}$ be an edge-signed graph. By a \emph{block
of $\mathcal{S}$} we mean the subgraph of $\mathcal{S}$ induced
by a block 
of $U(\mathcal{S})$.

\begin{lemma}\label{lm:B*3}
Let $\Ho$ be a fat Hoffman graph 
containing $\mathfrak{K}_{1,2}$ 
with $\lambda_{\min}(\Ho) > -3 $. 
Let $v^*$ be the slim vertex in the $\mathfrak{K}_{1,2}$, 
and let $\mathcal{B}^*$ be the block of $\mathcal{S}(\Ho)$ 
containing the vertex $v^*$. 
Then the block $\mathcal{B}^*$ is 
$\mathcal{K}^{+}_n$ $(n \geq 2)$, 
$\mathcal{K}^{-}_2$, or 
$\mathcal{T}_1^*$. 
\end{lemma}

\begin{proof}
If $E^-(\mathcal{B}^*) = \emptyset$, 
then $\mathcal{B}^* = \mathcal{K}^{+}_n$ 
with $n \geq 2$ 
since each block has at least two vertices. 
We assume that $E^-(\mathcal{B}^*) \neq \emptyset$. 
If $|V(\mathcal{B}^*)| = 2$, then $\mathcal{B}^* = \mathcal{K}^{-}_2$. 
If $|V(\mathcal{B}^*)| = 3$, 
then, by Lemma~\ref{lm:201b}, $\mathcal{B}^* = \mathcal{T}_1^*$ 
since $\lambda_{\min}(\Ho) > -3$. 

We show that $|V(\mathcal{B}^*)| \leq 3$ by contradiction. 
Suppose that $|V(\mathcal{B}^*)| \geq 4$. 
Take any three vertices $v_1$, $v_2$, $v_3$ in $\mathcal{B}^*$ 
other than $v^*$. Then by Lemma~\ref{lm:104c}, each of the
vertices $v_1$, $v_2$ and $v_3$ has exactly one fat neighbor.
Since $\lambda_{\min}(\Ho) > -3 $, 
it follows from Lemma~\ref{lm:201} that 
the edge-signed subgraph of $\mathcal{S}(\Ho)$ 
induced by 
$\{v_1, v_2, v_3\}$ is a $(+)$-triangle $\mathcal{K}_3^+$. 
Since $E^-(\mathcal{B}^*) \neq \emptyset$, 
without loss of generality, 
we may assume that $\{v^*, v_1 \}$ is a $(-)$-edge 
in $\mathcal{S}(\Ho)$. 
Since $\lambda_{\min}(\Ho)>-3$, Lemma~\ref{lm:201b} implies
that 
both of the edges 
$\{v^*, v_2 \}$ and $\{v^*, v_3 \}$ are $(-)$-edges. 
For $i=1,2,3$, the vertices 
$v^*$ and $v_i$ have a common fat neighbor, say $f_i$, in $\Ho$ 
since $\{v^*, v_i \}$ is a $(-)$-edge in $\mathcal{S}(\Ho)$. 
Since the vertex $v^*$ has exactly two fat neighbors in $\Ho$, 
two of the three fat vertices $f_1$, $f_2$, and $f_3$ are the same. 
Without loss of generality, we may assume that $f_1=f_2$. 
Then $v_1$ and $v_2$ have a common fat neighbor, 
which is a contradiction to the fact that 
$\{v_1, v_2 \}$ is a $(+)$-edge in $\mathcal{S}(\Ho)$. 
Thus $|V(\mathcal{B}^*)| \leq 3$. 
Hence the lemma holds. 
\end{proof}

Now we are ready to give our main result. 

\begin{theorem}\label{thm:101}
Let $\Ho$ be a fat indecomposable Hoffman graph 
containing 
a slim vertex $v^*$ having two fat neighbors. 
Then
$\lambda_{\min}(\Ho) > -3$ if and only if 
the following conditions hold:  
\begin{enumerate}[{\rm (i)}]
\item 
$U(\mathcal{S}(\Ho))$ is a claw-free block graph,
\item 
$\Ho$ has exactly one induced Hoffman subgraph isomorphic to 
$\mathfrak{K}_{1,2}$, 
\item 
$v^*$ is not a cut vertex of $U(\mathcal{S}(\Ho))$, 
\item 
the block $\mathcal{B}^*$ of $\mathcal{S}(\Ho)$ 
containing the vertex $v^*$ 
is either 
$\mathcal{K}^+_n$ 
$(n \geq 2)$ 
or 
$\mathcal{K}^-_2$ 
or $\mathcal{T}_1^*$, 
\item 
each block of $\mathcal{S}(\Ho)$ 
other than $\mathcal{B}^*$ is either 
$\mathcal{K}^+_n$ 
$(n \geq 2)$ 
or 
$\mathcal{K}^-_2$. 
\end{enumerate}
\end{theorem}

\begin{proof}
Suppose that $\lambda_{\min}(\Ho) > -3$.
Then (i), (ii), (iii), (iv), and (v) follow by Proposition~\ref{prop:s3},
Lemma~\ref{lm:104c}, 
Lemma~\ref{lm:103b}, 
Lemma~\ref{lm:B*3}, and
Lemma~\ref{lm:201}, 
respectively.

Conversely, assume that (i)--(v) hold.
Let $\{\mathcal{B}_0=\mathcal{B}^*, \mathcal{B}_1, \ldots, \mathcal{B}_p\}$ 
be the set of blocks of $\mathcal{S}(\Ho)$. 
For each block $\mathcal{B}$ with 
$\mathcal{B} \cong \mathcal{K}^-_2$, let
$V(\mathcal{B})=\{\sigma^+(\mathcal{B}),\sigma^-(\mathcal{B})\}$.
Let $W=\{w_1, \ldots, w_q \}$ be 
the set of slim vertices of $\Ho - v^*$ 
which 
are not cut vertices of $U(\mathcal{S}(\Ho))$. 
We define a map $\psi:V^s(\Ho) \to \mathbb{R}^N$, 
where $N=1+p+q$, by 
\[
\psi(v)_i = 
\left\{
\begin{array}{ll}
1 & \text{ if } i = 0,\; v=v^*,  \\
1 & \text{ if } i = 0,\; 
\mathcal{B}_0 = \mathcal{K}^+_n \text{ for some } n, \;
v \in V(\mathcal{B}_0) - \{v^*\}, \\
-1 & \text{ if } i = 0, \;
\mathcal{B}_0 = \mathcal{K}^-_2 \text{ or } \mathcal{T}^*_1,\;
 v \in V(\mathcal{B}_0) - \{v^*\}, \\
1 & \text{ if } 1 \leq i \leq p, \;
\mathcal{B}_i \cong \mathcal{K}^+_n \text{ for some $n$, }
v \in V(\mathcal{B}_i), \\
1 & \text{ if } 1 \leq i \leq p, \;
\mathcal{B}_i \cong \mathcal{K}^-_2,\; 
v=\sigma^+(\mathcal{B}_i), \\
-1 & \text{ if } 1 \leq i \leq p,\;
\mathcal{B}_i \cong \mathcal{K}^-_2,\; 
v=\sigma^-(\mathcal{B}_i), \\
1 & \text{ if } p+1 \leq i \leq p+q, \;
v = w_{i-p} \in W, \\ 
0 & \text{ otherwise. } 
\end{array}
\right.
\]
Then $\psi$ is a reduced representation of $\Ho$.

Next, we show that 
$\{\psi(v) \mid v \in V^s(\Ho)\}$ is linearly independent. 
Suppose that 
\[
\sum_{v \in V^s(\Ho)} a_v \psi(v) = 0 
\]
where $a_v \in \mathbb{R}$ for each $v \in V^s(\Ho)$. 
For each $1 \leq i \leq q$, 
$\psi(v)_{i+p} = \delta_{v, w_{i}}$, 
where $\delta_{v, w}$ denotes the Kronecker's delta. 
Since 
$\sum_{v \in V^s(\Ho)} a_v \psi(v)_{i+p} = 0$, 
we have 
$a_{w_{i}} = 0$. 
Thus $a_{v} = 0$ for any $v \in W$.

Suppose that there exists 
$u \in V^s(\Ho) - (W \cup \{v^*\})$ 
such that $a_u \neq 0$.  
We take $u$ as a vertex farthest from $v^*$ among such vertices. 
Since $u$ is a cut vertex, 
$u$ is contained in two blocks, say $\mathcal{B}_i$ and 
$\mathcal{B}_j$. 
Let $G=U(\mathcal{S}(\Ho))$ for convenience, denote by
$B^*=B_0,B_1,\dots,B_p$ the blocks in $G$ corresponding to
the blocks $\mathcal{B}^*=\mathcal{B}_0,\mathcal{B}_1,\dots,\mathcal{B}_p$,
respectively.
Note that $\{{B}_i, {B}_j\}$ 
is an edge in $\mathbb{B}(G)$. 
By (i) and Lemma~\ref{lm:block3}, 
$\mathbb{B}(G)$ is a tree. 
Therefore 
$d_{\mathbb{B}(G)}({B}^*, {B}_i)
=d_{\mathbb{B}(G)}({B}^*, {B}_j)\pm1$.
Without loss of generality, 
we may assume that
$d_{\mathbb{B}(G)}({B}^*, {B}_i)
=d_{\mathbb{B}(G)}({B}^*, {B}_j)+1$.
Let $P$ be the shortest path from $v^*$ to $u$. If $j\neq0$, then
$|V(P)\cap(V({B}^*)\cup V({B}_i))|=3$.
Thus by Lemma~\ref{lm:block4}, 
\[
d_{\mathbb{B}(G)}({B}^*, {B}_i) = d_G(v^*,u). 
\]
If $j=0$, then 
this 
holds as well, since both sides equal $1$.

Let $u'$ be any cut vertex in ${B}_i$ other than $u$, 
i.e., $u' \in V({B}_i) - (W \cup \{u\})$. 
Since $u$ is a cut vertex of $G$, 
the shortest path from $v^*$ to $u'$ 
must pass the vertex $u$. 
Therefore, 
$d_G(v^*,u') = d_G(v^*,u) + d_G(u,u') 
= d_{\mathbb{B}(G)}({B}^*, {B}_i) +1$. 
Then $a_{u'}=0$ by the choice of $u$. 
Then 
we obtain 
\begin{align*}
0 &= 
\sum_{v \in V^s(\Ho)} a_v \psi(v)_{i} 
\\&= 
a_u \psi(u)_{i} + 
\sum_{u' \in V({B}_i) - (W \cup \{u\})} a_{u'} \psi(u')_{i} + 
\sum_{v \in V({B}_i) \cap W} a_v \psi(v)_{i} 
\\&= 
\pm a_{u}, 
\end{align*}
which is a contradiction to 
$a_{u} \neq 0$. 
Thus, we have $a_{v} = 0$ for any 
$v \in V^s(\Ho) - (W \cup \{v^*\})$. 

Moreover, we obtain 
$0 = 
\sum_{v \in V^s(\Ho)} a_v \psi(v)_{0} 
= a_{v^*}$. 
Thus we have 
$a_{v} = 0$ for any $v \in V^s(\Ho)$. 
Hence 
$\{\psi(v) \mid v \in V^s(\Ho)\}$ is linearly independent. 
By Lemma~\ref{lm:redrep2}, 
$\lambda_{\min}(\Ho) > -3 $. 
\end{proof}

\begin{remark}
In the proof of Theorem~\ref{thm:101}, we constructed a reduced
representation of norm $3$ of the Hoffman graph $\Ho$ satisfying
(i)--(v) with integral entries. In general, a reduced representation
may not be realizable in $\mathbb{Z}^n$, but it is shown in
\cite[Theorem 2.8]{JKMT} that a graph satisfying the conditions
of Theorem~\ref{thm:101} admits such a reduced representation.
\end{remark}

\section{Concluding remarks}\label{sec:conc}

Our main theorem gives a characterization of fat indecomposable
Hoffman graphs $\Ho$ with $\lambda_{\min}(\Ho) > -3$
containing a slim vertex $v^*$ having two fat neighbors,
in terms of their special graphs. This is natural, since
the smallest eigenvalue of $\Ho$ is determined only by
its special graph $\mathcal{S}(\Ho)$. Given a 
connected edge-signed graph
$\mathcal{S}$ satisfying the following conditions:
\begin{enumerate}[{\rm (i)}]
\item 
$U(\mathcal{S})$ is a claw-free block graph,
\item 
No vertex other than $v^*$ is incident with more than one $(-)$-edge,
\item 
$v^*$ is not a cut vertex of $U(\mathcal{S})$, 
\item 
the block $\mathcal{B}^*$ of $\mathcal{S}$ containing the vertex $v^*$ 
is either 
$\mathcal{K}^+_n$ $(n \geq 2)$ or $\mathcal{K}^-_2$ or $\mathcal{T}_1^*$, 
\item 
each block of $\mathcal{S}$ other than $\mathcal{B}^*$ is either 
$\mathcal{K}^+_n$ $(n \geq 2)$ or $\mathcal{K}^-_2$,
\end{enumerate}
one can construct a fat Hoffman graph $\Ho$ with $\mathcal{S}(\Ho)=\mathcal{S}$,
such that $v^*$ is the only slim vertex having two fat neighbors.
Indeed, $\Ho$ can be constructed from $\mathcal{S}$ in the following manner:
\begin{enumerate}[{\rm (a)}]
\item
for each $(-)$-edge,
attach a common fat neighbor to its end vertices,
\item
if $\mathcal{B}^*\cong\mathcal{K}_n^+$, then attach two pendant fat
vertices to $v^*$;
if $\mathcal{B}^*\cong\mathcal{K}_2^-$, then attach a pendant fat
vertex to $v^*$,
\item
for each vertex other than $v^*$ which is not incident to a $(-)$-edge,
attach a pendant fat vertex,
\item
replace every $(+)$-edge of $\mathcal{S}$ by an edge, 
and remove all $(-)$-edges of $\mathcal{S}$.
\end{enumerate}
Then one can verify that $\Ho$ is a fat Hoffman graph with $\mathcal{S}(\Ho)
=\mathcal{S}$, and $v^*$ has two fat neighbors.

It should be remarked, however, that a fat Hoffman graph $\Ho$
with prescribed $\mathcal{S}(\Ho)$ is not unique. Indeed, let
$\mathcal{S}$ be the path of length $2$ consisting of two $(+)$-edges,
with vertex set
$\{v^*,v_1,v_2\}$, where $v^*$ is an end vertex. Then the two
Hoffman graphs $\Ho^i$ $(i=1,2)$ defined below
satisfy $\mathcal{S}(\Ho^i)=\mathcal{S}$.
\begin{align*}
V^f(\Ho^1)&=\{f_+,f_-,f_1,f_2\},\\
E(\Ho^1)&=\{\{v^*,v_1\},\{v_1,v_2\},\{v^*,f_+\},\{v^*,f_-\},\{v_1,f_1\},\{v_2,f_2\}\},\\
V^f(\Ho^2)&=\{f_0,f_1,f_2\},\\
E(\Ho^2)&=\{\{v^*,v_1\},\{v^*,v_2\},\{v_1,v_2\},
\{v^*,f_0\},\{v^*,f_2\},\{v_1,f_1\},\{v_2,f_2\}\}.
\end{align*}

Our main theorem also gives a generalization of a result of 
Hoffman~\cite{hoffman0}.
Recall that 
$\hat{A}(G,v^*)$ denotes the adjacency matrix of a graph $G$,
modified by putting $-1$ in the diagonal position corresponding to 
a vertex $v^*$. 
As we mentioned in Section \ref{sec:001}, 
Hoffman showed the following. 

\begin{lemma}[{\cite[Lemma 2.1]{hoffman0}}] \label{lm:H0}
Let $L(T)$ be the line graph of a tree $T$ 
and let $e$ be an end edge of $T$. 
Then 
the smallest eigenvalue of $\hat{A}(L(T),e)$ is greater than $-2$. 
\end{lemma}

We can generalize Lemma~\ref{lm:H0} by using
Theorem~\ref{thm:101}.  

\begin{theorem}\label{thm:5}
Let $G$ be a graph and let $v^*$ be a vertex of $G$. 
Then the smallest eigenvalue of $\hat{A}(G,v^*)$ is greater than $-2$ 
if and only if $G$ is the line graph of a tree $T$
and $v^*$ corresponds to an end edge of $T$. 
\end{theorem}
\begin{proof}
Suppose that
the smallest eigenvalue of $\hat{A}(G,v^*)$ is greater than $-2$.
Let $\Ho$ be the fat Hoffman graph obtained by attaching a 
pendant fat vertex to
every vertex of $G$ except $v^*$, and attaching two pendant 
fat vertices
to $v^*$. Then $B(\Ho)=\hat{A}(G,v^*)-I$, hence $\Ho$ has smallest
eigenvalue greater than $-3$. By Theorem~\ref{thm:101}(i), 
$U(\mathcal{S}(\Ho))$ is a claw-free block graph, hence
it is a line graph of a tree by Lemma~\ref{lm:block3}.
Moreover, by Theorem~\ref{thm:101}(iii), $v^*$ is not a cut vertex
of $U(\mathcal{S}(\Ho))$. Thus $v^*$ corresponds to an end edge of $T$.

Conversely, suppose that $G$ is the line graph of a tree $T$
and $v^*$ corresponds to an end edge of $T$. It is easy to see that
the Hoffman graph $\Ho$ constructed above satisfies all the 
conditions (i)--(v) of Theorem~\ref{thm:101}, 
hence $\lambda_{\min}(\Ho)>-3$. Since
$B(\Ho)=\hat{A}(G,v^*)-I$, we conclude that
$\lambda_{\min}(\hat{A}(G,v^*))>-2$. 
\end{proof}


\newpage 
\section*{Appendix: Detailed proofs of Lemmas 
\ref{lm:Dn}, \ref{lm:103a}, and \ref{lm:-3c}}

\newcommand{\nexteq}{\displaybreak[0]\\ &=}
\newcommand{\bx}{\mathbf{x}}

\begin{proof}[Proof of Lemma \ref{lm:Dn}]
By assumption, we have 
\[ 
B(\Ho)=
\begin{pmatrix}
-2 & \epsilon_1 & & & & &  \\
\epsilon_1 & -1 & \epsilon_2 & & & & & \\
 & \epsilon_2 & -1 & \ddots & & & & \\
 & & \ddots & \ddots & \epsilon_{n-3} & & \\
 & & & \epsilon_{n-3} & -1 & \epsilon_{n-2} & \epsilon_{n-1}  \\
 & & & & \epsilon_{n-2} & -1 & 0 \\
 & & & & \epsilon_{n-1} & 0 & -1 \\
\end{pmatrix},
\]
Let ${\bf x} \in \mathbb{R}^n$ be a vector defined by 
\[
({\bf x})_i=
\begin{cases}
2 \prod_{k=i}^{n-3} (-\epsilon_k) & \text{if } 1 \leq i \leq n-3,\\ 
2 & \text{if } i= n-2, \\
- \epsilon_{i-1} & \text{if } i= n-1, n. 
\end{cases}
\]
Then $ (B(\Ho)+3I){\bf x}={\bf 0} $.
Indeed,
\begin{align*}
(B\bx)_1 &= -2(\bx)_1+\epsilon_1(\bx)_2
=
-4\prod_{k=1}^{n-3}(-\epsilon_k)
+2\epsilon_1\prod_{k=2}^{n-3}(-\epsilon_k)
\nexteq
-6\prod_{k=1}^{n-3}(-\epsilon_k)
= 
-3(\bx)_1.
\end{align*}
For $2\leq i\leq n-3$,
\begin{align*}
(B\bx)_i&=\epsilon_{i-1}(\bx)_{i-1}-(\bx)_i+\epsilon_i(\bx)_{i+1}
\nexteq
2\epsilon_{i-1}\prod_{k=i-1}^{n-3}(-\epsilon_k)
-2\prod_{k=i}^{n-3}(-\epsilon_k)
+2\epsilon_{i+1}\prod_{k=i+1}^{n-3}(-\epsilon_k)
\nexteq
-2\prod_{k=i}^{n-3}(-\epsilon_k)
-2\prod_{k=i}^{n-3}(-\epsilon_k)
-2\prod_{k=i}^{n-3}(-\epsilon_k)
\nexteq
-6\prod_{k=i}^{n-3}(-\epsilon_k)
= -3(\bx)_i.
\end{align*}
\begin{align*}
(B\bx)_{n-2} &= \epsilon_{n-3}(\bx)_{n-3}-(\bx)_{n-2}+
\epsilon_{n-2}(\bx)_{n-1}+\epsilon_{n-1}(\bx)_n
\nexteq
2\epsilon_{n-3}(-\epsilon_{n-3})-2-\epsilon_{n-2}^2-\epsilon_{n-1}^2
= -6
= -3(\bx)_{n-2},
\end{align*}
\[ 
(B\bx)_{n-1} 
= \epsilon_{n-2}(\bx)_{n-2}-(\bx)_{n-1}
= 2\epsilon_{n-2}+\epsilon_{n-2} 
= 3\epsilon_{n-2}
= -3(\bx)_{n-1},
\]
\[
(B\bx)_{n} 
= \epsilon_{n-1}(\bx)_{n-2}-(\bx)_{n}
= 2\epsilon_{n-1}+\epsilon_{n-1}
= 3\epsilon_{n-1}
= -3(\bx)_{n}.
\]
Hence $ B(\Ho) $ has $ -3 $ as an eigenvalue.
This implies $ \lambda_{\min}(\Ho)\le -3 $. 
\end{proof}

\begin{proof}[Proof of Lemma \ref{lm:103a}]
Consider the matrix
\[ 
B(\Ho)=
\begin{pmatrix}
-2 & \epsilon_1 & 0 & \cdots & \cdots & 0 & 0 & 0 & 0 \\
\epsilon_1 & -1 & \epsilon_2 & \ddots & & \vdots & \vdots & \vdots & \vdots  \\
0 & \epsilon_2 & -1 & \ddots & \ddots & \vdots& & &  \\
\vdots & \ddots & \ddots & \ddots & \epsilon_{n-2} & 0 &  \vdots & \vdots  & \vdots   \\
\vdots & & \ddots & \epsilon_{n-2} & -1 & \epsilon_{n-1} & 0& 0& 0 \\
0 & \cdots & \cdots & 0 & \epsilon_{n-1} & -1 & 1 & 1 & \epsilon \\
0 & \cdots & & \cdots & 0& 1 & -1 & \delta & 1 \\
0 & \cdots & & \cdots & 0& 1 & \delta & -1 & 1 \\
0 & \cdots & & \cdots & 0& \epsilon  & 1 & 1 & -1 \\
\end{pmatrix},
\]
where $\epsilon_1, \dots, \epsilon_{n-1} \in \{ \pm1 \}$ and
$\{ \epsilon, \delta\} = \{ 0,1 \}$. 
Let $\mathbf{x} \in \mathbb{R}^{n+3}$ be a vector defined by 
\[
	(\mathbf{x})_i=\begin{cases}
	2\prod_{k=i}^{n-1}(- \epsilon_k) & \text{if } 1 \leq i \leq n-1,\\
	2 & \text{if } i=n,\\
	-1 & \text{if } i=n+1,n+2,\\
	2(\epsilon-1)/(\epsilon-2)&\text{if } i=n+3. 
	\end{cases}
\]
We claim $B(\Ho)\bx=-3\bx$. 
Indeed,
\begin{align*}
(B\bx)_{1}&=-2(\bx)_1+\epsilon_1(\bx)_2 \\
&=
-2(-1)^{n-1}\cdot 2\prod_{k=1}^{n-1}\epsilon_k
+\epsilon_1(-1)^{n-2}\cdot 2\prod_{k=2}^{n-1}\epsilon_k
\\&=
-4(-1)^{n-1}\prod_{k=1}^{n-1}\epsilon_k
-2(-1)^{n-1}\prod_{k=1}^{n-1}\epsilon_k
\\&=
-3(-1)^{n-1}2\prod_{k=1}^{n-1}\epsilon_k
=-3(\bx)_1. 
\end{align*}
For $2 \leq j \leq n-1$, 
\begin{align*}
(B\bx)_j&=
\epsilon_{j-1}(-1)^{n-j+1}2\prod_{k=j-1}^{n-1}\epsilon_k
-(-1)^{n-j}2\prod_{k=j}^{n-1}\epsilon_k
+\epsilon_j(-1)^{n-j-1}2\prod_{k=j+1}^{n-1}\epsilon_k
\nexteq
2\prod_{k=j}^{n-1}\epsilon_k
((-1)^{n-j+1}-(-1)^{n-j}+(-1)^{n-j-1})
\nexteq
-3\cdot(-1)^{n-j}2\prod_{k=j}^{n-1}\epsilon_k
= 
-3(\bx)_j.
\end{align*}
\begin{align*}
(B\bx)_{n}&=
\epsilon_{n-1}(\bx)_{n-1}-(\bx)_n+(\bx)_{n+1}+(\bx)_{n+2}+\epsilon(\bx)_{n+3}
\nexteq
\epsilon_{n-1}(-2\epsilon_{n-1})-2+(-1)+(-1)+
\frac{2\epsilon(\epsilon-1)}{\epsilon-2}
\nexteq
-6
= -3(\bx)_{n},
\end{align*}
\begin{align*}
(B\bx)_{n+1}&=
(\bx)_n-(\bx)_{n+1}+(1-\epsilon)(\bx)_{n+2}+(\bx)_{n+3}
\nexteq
2-(-1)+(1-\epsilon)(-1)+\frac{2(\epsilon-1)}{\epsilon-2}
\nexteq
3+\frac{(\epsilon-1)((\epsilon-2)+2)}{\epsilon-2}
= 
-3(\bx)_{n+1},
\end{align*}
\begin{align*}
(B\bx)_{n+2}&=
(\bx)_n+(1-\epsilon)(\bx)_{n+1}-(\bx)_{n+2}+(\bx)_{n+3}
\nexteq
2+(1-\epsilon)(-1)-(-1)+\frac{2(\epsilon-1)}{\epsilon-2}
\nexteq
3+\frac{(\epsilon-1)((\epsilon-2)+2)}{\epsilon-2}
= 
3
= 
-3(\bx)_{n+2},
\end{align*}
\begin{align*}
(B\bx)_{n+3}&=
\epsilon(\bx)_n+(\bx)_{n+1}+(\bx)_{n+2}-(\bx)_{n+3}
\nexteq
2\epsilon-1-1-\frac{2(\epsilon-1)}{\epsilon-2}
= 
2(\epsilon-1)-\frac{2(\epsilon-1)}{\epsilon-2}
\nexteq
\frac{(2(\epsilon-2)-2)(\epsilon-1)}{\epsilon-2}
= 
\frac{-6(\epsilon-1)}{\epsilon-2}
+\frac{2\epsilon(\epsilon-1)}{\epsilon-2}
\nexteq
\frac{-6(\epsilon-1)}{\epsilon-2}
= 
-3(\bx)_{n+3}.
\end{align*}
Thus $B(\Ho)\bx+3\bx=\mathbf{0}$.
\end{proof}

\begin{proof}[Proof of Lemma \ref{lm:-3c}]
Consider the matrix
\[ 
\left(\begin{array}{cccccccccccc}
-2 & \delta_1 & 0 & \cdots & \cdots & 0 & 0 & 0 & \cdots & & \cdots & 0 \\
\delta_1 & -1 & \delta_2 & \ddots & & \vdots & \vdots & \vdots & & & & \vdots  \\
0 & \delta_2 & -1 & \ddots & \ddots & \vdots  &&&&&& \\
\vdots & \ddots & \ddots & \ddots & \delta_{m-2} & 0 & \vdots & \vdots &&&& \vdots \\
\vdots & & \ddots & \delta_{m-2} & -1 & \delta_{m-1} & 0 & 0 & \cdots & & \cdots & 0\\
0 & \cdots & \cdots & 0 & \delta_{m-1} & -1 & 1 & 1 & 0 & \cdots & \cdots & 0 \\
0 & \cdots && \cdots & 0 & 1& -1 & 1 & 0 & \cdots & 0 & \epsilon_n \\
0 & \cdots && \cdots & 0 & 1& 1 & -1 & \epsilon_2 & \ddots & & 0  \\
\vdots & & & & \vdots & 0 & 0 & \epsilon_2 & -1 & \ddots & \ddots & \vdots  \\
& & & & & \vdots & \vdots & \ddots & \ddots & \ddots & \epsilon_{n-2} & 0 \\
\vdots &&&& \vdots & \vdots &  0 & & \ddots & \epsilon_{n-2} & -1 & \epsilon_{n-1}\\
0 & \cdots & & \cdots & 0 & 0 & \epsilon_n & 0 & \cdots & 0 & \epsilon_{n-1} & -1
\end{array}\right),
\]
where 
$\delta_1,\dots,\delta_{m-1}, \epsilon_2,\dots,\epsilon_n\in\{\pm1\}$. 
Since the number of $(+)$-edges in $\mathcal{C}_n$ is odd, 
we have 
\[
\prod_{k=2}^n\epsilon_k=(-1)^{n-1}.
\]
Let $\mathbf{x} \in \mathbb{R}^{m+n}$ be the vector defined by 
\[
(\mathbf{x})_i = 
\left\{
\begin{array}{ll}
2 
& \text{if } i=1, \\ 
2 \cdot \prod_{k=1}^{i-1}(-\delta_k) 
& \text{if } 2 \leq i \leq m, \\
- \prod_{k=1}^{m-1}(-\delta_k) 
& \text{if } m+1\leq i\leq m+2, \\
-\prod_{k=1}^{m-1}(-\delta_k) \cdot \prod_{k=2}^{i-m-1}(-\epsilon_k) 
& \text{if } m+3 \leq i \leq m+n. 
\end{array} 
\right. 
\]
Then $(B(\Ho)+3I) \mathbf{x}={\bf 0}$. Indeed,
set
$\delta := \prod_{k=1}^{m-1}(-\delta_k)$.
Then
\begin{align*}
(B(\Ho)\mathbf{x})_1&=
-2(\mathbf{x})_1+\delta_1(\mathbf{x})_2
= 
-2\cdot2+\delta_1\cdot(2\cdot(-\delta_1))
\nexteq
-4-2\delta_1^2
= -6
= -3(\mathbf{x})_1. 
\end{align*}
For $2\leq j\leq m-1$, 
\begin{align*}
(B(\Ho)\mathbf{x})_j&=
\delta_{j-1}(\mathbf{x})_{j-1}
-(\mathbf{x})_{j}
+\delta_j(\mathbf{x})_{j+1}
\nexteq
\delta_{j-1}\cdot2\cdot\prod_{k=1}^{j-2}(-\delta_k)
-2\cdot\prod_{k=1}^{j-1}(-\delta_k)
+\delta_j\cdot2\cdot\prod_{k=1}^{j}(-\delta_k)
\nexteq
(-2-2-2)\prod_{k=1}^{j-1}(-\delta_k)
= 
-3(\mathbf{x})_{j}.
\end{align*}
\begin{align*}
(B(\Ho)\mathbf{x})_m&=
\delta_{m-1}(\mathbf{x})_{m-1}-(\mathbf{x})_{m}
+(\mathbf{x})_{m+1}+(\mathbf{x})_{m+2}
\nexteq
\delta_{m-1}\cdot2\cdot\prod_{k=1}^{m-2}(-\delta_k)
-2\delta
-\delta-\delta
\nexteq
(-2-2-1-1)\delta
= 
-3(\mathbf{x})_m, 
\end{align*}
\begin{align*}
(B(\Ho)\mathbf{x})_{m+1}&=
(\mathbf{x})_m-(\mathbf{x})_{m+1}+(\mathbf{x})_{m+2}
+\epsilon_n(\mathbf{x})_{m+n}
\nexteq
2\delta
-(-\delta)+(-\delta)
+\epsilon_n(-\delta)
\prod_{k=2}^{n-1}(-\epsilon_k)
\nexteq
2\delta
+\delta
\prod_{k=2}^{n}(-\epsilon_k)
= 
2\delta
+\delta
= 
-3(-\delta)
= 
-3(\mathbf{x})_{m+1}, 
\end{align*}
\begin{align*}
(B(\Ho)\mathbf{x})_{m+2}&=
(\mathbf{x})_{m}+(\mathbf{x})_{m+1}
-(\mathbf{x})_{m+2}+\epsilon_2(\mathbf{x})_{m+3}
\nexteq
2\delta -\delta +\delta +\epsilon_2(-\delta(-\epsilon_2))
\nexteq
2\delta +\delta
= 
-3(-\delta)
= 
-3(\mathbf{x})_{m+2}.
\end{align*}
For $m+3\leq j\leq m+n-1$,
\begin{align*}
(B(\Ho)\mathbf{x})_j&=
\epsilon_{j-m-1}(\mathbf{x})_{j-1}-(\mathbf{x})_j
+\epsilon_{j-m}(\mathbf{x})_{j+1}
\nexteq
\epsilon_{j-m-1}(-\delta
\prod_{k=2}^{j-m-2}(-\epsilon_k))
-(-\delta
\prod_{k=2}^{j-m-1}(-\epsilon_k))
+\epsilon_{j-m}(-\delta\prod_{k=2}^{j-m}(-\epsilon_k))
\nexteq
(-1-1-1)
(-\delta
\prod_{k=2}^{j-m-1}(-\epsilon_k))
= 
-3(\mathbf{x})_j.
\end{align*}
\begin{align*}
(B(\Ho)\mathbf{x})_{m+n}&=
\epsilon_n(\mathbf{x})_{m+1}
+\epsilon_{n-1}(\mathbf{x})_{m+n-1}
-(\mathbf{x})_{m+n}
\nexteq
\epsilon_n(-\delta)
+\epsilon_{n-1}(-\delta
\prod_{k=2}^{n-2}(-\epsilon_k))
-(-\delta\prod_{k=2}^{n-1}(-\epsilon_k))
\nexteq
\epsilon_n\prod_{k=2}^n(-\epsilon_k)\cdot
(-\delta)
+\delta
\prod_{k=2}^{n-1}(-\epsilon_k)
+\delta
\prod_{k=2}^{n-1}(-\epsilon_k)
\nexteq
3\delta
\prod_{k=2}^{n-1}(-\epsilon_k)
= 
-3(\mathbf{x})_{m+n}.
\end{align*}
Therefore, $\Ho$ has $-3$ as an eigenvalue 
and thus 
$ \lambda_{\min}(\Ho) \leq -3 $. 
\end{proof}

\end{document}